\documentclass[3p]{elsarticle}

\usepackage{amsfonts}
\usepackage{graphicx}

\newtheorem{theorem}{Theorem}[section]
\newtheorem{corollary}[theorem]{Corollary}

\newtheorem{example}{Example}

\newtheorem{observation}[theorem]{Observation}
\newtheorem{problem}[theorem]{Problem}

\newtheorem{remark}[theorem]{Remark}
\newtheorem{algorithm}{Algorithm}
\newenvironment{proof}[1][Proof]{\textbf{#1.} }{\ \rule{0.5em}{0.5em}}

\begin{document}

\title{ An algorithm for constructing doubly stochastic matrices for the inverse eigenvalue problem\tnoteref{t1}}
\tnotetext[t1]{This work is supported by the Lebanese university
research grants program}
\author[rvt]{Bassam Mourad\corref{cor1}\fnref{fn1}}
\ead{bmourad@ul.edu.lb}
\author[focal]{ Hassan Abbas}
\author[focal]{Ayman Mourad }
\author[focal]{Ahmad Ghaddar }
\author[els]{ Issam Kaddoura }
\cortext[cor1]{Corresponding author}
 \fntext[fn1]{Tel.:+961 3 784363; Fax:+961 7 768174.}
\address[rvt]{Department of Mathematics, Faculty of Science V, Lebanese University, Nabatieh,
Lebanon}
\address[focal]{  Department of
Mathematics,  Faculty of Science, Lebanese University, Beirut,
Lebanon}
\address[els]{  Department of
Mathematics,  Faculty  of Arts and Science, Lebanese International
University, Saida, Lebanon}

\begin{abstract}
In this note, we present an algorithm that yields many new methods
for constructing doubly stochastic and symmetric doubly stochastic
matrices for the inverse eigenvalue problem. In addition, we
introduce new open problems in this area that lay the ground for
future work.
\end{abstract}


\maketitle

\section{Introduction}

An $n\times n$ matrix with real entries is said to be
\textit{nonnegative } if all of its entries are nonnegative. An
$n\times n$ matrix $A$ over a field $\mathbf{F}$ ($\mathbf{F}$ is
either the real line $\mathbf{R}$ or the complex plane
$\mathbf{C})$ having each row and column sum equal to $r$, is said
to be an \textit{$r$-generalized doubly stochastic} matrix. The
set of all $n\times n$ $r$-generalized doubly stochastic matrices
with entries in $\mathbf{F}$ is denoted by $\Omega
^{r}(n,\mathbf{F})$. A \textit{generalized doubly stochastic}
matrix is an element of $\Omega (n,\mathbf{F})$ where
\[
\Omega (n,\mathbf{F})=\bigcup_{r\in \mathbf{F}}\Omega ^{r}(n,\mathbf{F}).
\]
Of special importance are the \emph{\ nonnegative} elements in
$\Omega ^{1}(n,\mathbf{R})$ which are called the \emph{doubly
stochastic} matrices. The theory of doubly stochastic matrices is
particularly endowed with a large collection of applications in
other area of mathematics and also in other disciplines (see for
example ~\cite{b:bap,i:in,i:ing,b:bru,c:chu,m:mo,s:sen,z:zy}). Let
the set of all $n\times n$ doubly stochastic matrices be denoted
by $\Delta _{n}$ and the set of all $n\times n$ symmetric elements
in $\Delta _{n}$ will be denoted by $ \Delta _{n}^{s}$. In
addition, let $\mathbf{M}(n,\mathbf{F})$ be the algebra of all $
n\times n$ matrices with entries in $\mathbf{F}$ and
$\mathbf{GL}(n,\mathbf{F} )$ be the \emph{general linear group
}over the field $\mathbf{F.}$ Finally, $I$ is defined as the
imaginary unit and for any matrix (or vector) $A$, its transpose
will be denoted by $A^T.$

The Perron-Frobenius theorem states that if $A$ is a nonnegative
matrix, then it has a real eigenvalue $r$ (that is the
Perron-Frobenius root) which is greater than or equal to the
modulus of each of the other eigenvalues. Also, $A$ has an
eigenvector $x$ corresponding to $r$ such that each of its entries
are nonnegative. Furthermore, if $A$ is irreducible then $r$ is
positive and the entries of $x$ are also positive (see
~\cite{b:ber,b:bha,h:hor,m:min,s:sen}). For doubly stochastic
matrices, $r=1$ and $x=e_{n}=(1,1,...,1)^T.$

An intriguing object of study in the area of matrix theory and
mathematical physics is that of the spectral properties and
inverse eigenvalue problems for special kinds of matrices. The
\textit{nonnegative inverse eigenvalue problem} (NIEP) is the
problem of finding necessary and sufficient conditions for an
$n$-tuples $ (\lambda_{1},\ldots,\lambda_{n})$ (possibly complex)
to be the spectrum of an $n\times n$ nonnegative matrix $A.$
Although this inverse eigenvalue problem has attracted a
considerable amount of interest, for $n>3$ it is still unsolved
except in restricted cases. Generally, we have two cases that
result in three problems.
\begin{itemize}
\item When $\lambda=(\lambda_{1},...,\lambda_{n})$ is complex,
little is known. The case $n\leq3$ was completely solved
in~\cite{l:loe} and the solution of the $4\times4$ trace-zero
nonnegative inverse eigenvalue problem was given in~\cite{r:re}.
 \item
 When the $n$-tuples $(\lambda_{1},\ldots,\lambda_{n})$ are real,
  then we have the following two problems:
\newline $1)$ The \textit{real nonnegative inverse eigenvalue
problem} (RNIEP) asks which sets of $n$ real numbers occur as the
spectrum of an $n\times n$ nonnegative matrix $A$.
\newline $2)$
The \textit{symmetric nonnegative inverse eigenvalue problem}
(SNIEP) asks which sets of $n$ real numbers can occur as the
spectrum of an $n\times n$ symmetric nonnegative matrix $A$.
\end{itemize}
Each problem remains open. Many partial results for the three
problems are known see the recent book~\cite{h:hog} and the
references therein, for the collection of all known results
concerning these problems. Various people raised the question
whether the (RNIEP) and (SNIEP) are generally equivalent. In the
low dimension $n\leq4,$ the two problems are actually equivalent
(see~\cite{e:eg}). However, for $n>4$, the paper~\cite{l:lo}
showed that the two problems are generally different. In addition,
the paper~\cite{e:eg} gives a construction for data
$\lambda=(\lambda_{1},..., \lambda_{5})$ which is a solution for
the (RNIEP) and there is no symmetric nonnegative $5\times 5$
matrix with spectrum $\lambda$.

 Another object of study in this area that has a big
interest is the inverse eigenvalue problem for nonnegative
matrices with extra properties. For example, we can consider the
same problems for doubly stochastic matrices. So that we have the
following problems.
 \begin{problem} The doubly stochastic inverse
eigenvalue problem denoted by (DIEP),  is the problem of
determining the necessary and sufficient conditions for a complex
$n$-tuples to be the spectrum of an $n\times n$ doubly stochastic
matrix. Equivalently, this problem can also be characterized as
the problem of finding the region $\Theta_n$ of $\mathbf{C}^n$
such that any point in $\Theta_n$ is the spectrum of an $n\times
n$ doubly stochastic matrix.
\end{problem}
Now, when the $n$-tuples $(\lambda_{1},\ldots,\lambda_{n})$ are
all real, then we have the following two problems:
 \begin{problem} The real
doubly stochastic inverse eigenvalue problem (RDIEP) asks which
sets of $n$ real numbers occur as the spectrum of an $n\times n$
doubly stochastic matrix. This problem is also equivalent to the
problem of finding the region $\Theta_n^r$ of $\mathbf{R}^n$ such
that any point in $\Theta_n^r$ is the spectrum of an $n\times n$
doubly stochastic matrix.
\end{problem}

\begin{problem} The symmetric doubly stochastic inverse eigenvalue
problem (SDIEP) asks which sets of $n$ real numbers occur as the
spectrum of an $n\times n$ symmetric doubly stochastic matrix.
Equivalently, this problem can also be characterized as the
problem of finding the region $\Theta_n^s$ of $\mathbf{R}^n$ such
that any point in $\Theta_n^s$ is the spectrum of an $n\times n$
symmetric doubly stochastic matrix.
\end{problem}
\begin{example}
 The point $\alpha=(1,-1/2+I\sqrt{3}/2,-1/2-I\sqrt{3}/2)$ is in $\Theta_3$ as $\alpha$ is
the spectrum of $A=\left(
\begin{array}{ccc}
0 & 1 & 0 \\
0 & 0 & 1 \\
1 & 0 & 0 \\
&
\end{array}
\right).$ On the other hand, the point $\lambda=(1,1/2,1/4)$ is in
$\Theta_3^r$ as $\lambda$ is the spectrum of $B=\left(
\begin{array}{ccc}
7/12 & 1/6 & 1/4 \\
1/12 & 2/3 & 1/4 \\
1/3 & 1/6 & 1/2 \\
&
\end{array}
\right).$ Moreover, $\lambda$ is in $\Theta_3^s$ as $\lambda$ is
also the spectrum of $C=\left(
\begin{array}{ccc}
13/24 & 7/24 & 1/6 \\
7/24 & 13/24 & 1/6 \\
1/6 & 1/6 & 2/3 \\
&
\end{array}
\right).$
\end{example}
 Although doubly stochastic matrices have been studied extensively,
 the (DIEP) and (RDIEP) have been considered in ~\cite{p:per,s:sot} where
 all the results obtained are partial.  The (SDIEP)
was studied in~\cite{h:hw,m:mo,m:mou,m:mour,r:rea}, and earlier
work can be found in~\cite{j:jo,p:per,s:sou} and all the results
obtained are also partial. For general $n,$ all three problems
remain open. In addition, the first and last of these problems
have been completely solved for $n=3$ in~\cite{p:per}. Now in a
similar manner to the (NIEP), one could raise the question whether
the (RDIEP) and (SDIEP) are generally equivalent. For $n=3,$ we
shall prove below that the two problems are actually equivalent.
While for $n\geq 4$, it remains a very interesting open problem.
In this paper, we are only able to obtain partial solutions
concerning all 3 problems.

Recall that $X\in \Omega ^{r}(n,\mathbf{F})$ if and only if
$Xe_{n}=re_{n}$ and $e_{n}^{T}X=re_{n}^{T}$ if and only if
$XJ_{n}=J_{n}X=rJ_{n},$ where $J_n$ is the $n\times n$ matrix with
each of its entries is equal to $\frac{1}{n}$. Hence for any $X\in
\Omega ^{r}(n,\mathbf{F})$, $e_{n}$ is an eigenvector for both $X$
and $X^{T}$ corresponding to the eigenvalue $r$. In addition, for
any $V$ in $\mathbf{GL}(n,\mathbf{F})$ such that the first row of
$V$ is equal to a multiple of $e_{n}^{T},$
 and each of the last $n-1$ row sums is zero, then clearly, its inverse
 $V^{-1}$ has its first column equal to a multiple of $e_{n}$ and each of the last $n-1$
column sums of $ V^{-1}$ is zero. Such $V$ is said to
have\textit{\  pattern} ${\mathcal{S}}.$ Note that any $n\times n$
matrix $V$ which is orthogonal and has its first column
$\frac{1}{\sqrt{n}}e_{n}$ has pattern ${\mathcal{S}}$. The main
results in this paper rely on the following observation for which
the proof can be found in~\cite{m:mou}.
 \begin{observation} Let
$V\in \mathbf{GL}(n,\mathbf{F})$ has pattern ${\mathcal{S}}$ and
$X\in \mathbf{M} (n-1,\mathbf{F})$, then $V^{-1}\left(
\begin{array}{cc}
r & 0 \\
0 & X \\
&
\end{array}
\right) V=A\in \Omega ^{r}(n,\mathbf{F}).$ Conversely, for any $A\in \Omega
^{r}(n,\mathbf{F})$ and any $V$ that has pattern $\mathcal{S}$, there exists
$X\in \mathbf{M}(n-1,\mathbf{F})$ such that $VAV^{-1}=\left(
\begin{array}{cc}
r & 0 \\
0 & X \\
&
\end{array}
\right) .$
\end{observation} In addition, an obvious necessary conditions for
  all three above problems are \begin{itemize}
\item
$\sum_{i=1}^{n}\lambda_{i}^{k}\geq0 \mbox{ for all natural number
}k
$
 which just means that the trace of the nonnegative matrix $A^{k}$ is
  nonnegative. \item $0\leq |\lambda_i| \leq 1$ for
  $i=2,...,n;$ as the Perron-Frobenius theorem insures.
  \item If $\lambda_i$ is a complex eigenvalue of a doubly stochastic
  matrix $D$ with nonzero imaginary part, then its conjugate $\bar{\lambda}_i$
  is also an eigenvalue of $D.$
  \end{itemize} Finally, we end this section with the following notation which is
used throughout this paper. Let $ \Lambda $ be the $n\times n$
diagonal matrix with diagonal entries $ 1,\lambda _{2},...,\lambda
_{n}$ with $1\geq \lambda _{2}\geq ...\geq \lambda _{n}\geq -1$
and $ trace(\Lambda)= 1+ \lambda _{2}+ ...+ \lambda _{n}\geq 0.$

\section{Construction of Matrices for the SDIEP}

In this section, we describe a way of obtaining partial solutions
for the (SDIEP). As mentioned earlier, this problem is completely
solved only for $n=3$. Indeed the following theorem is proved
in~\cite{p:per}.
\begin{theorem}~\cite{p:per}\label{th:th1}
 There exists a symmetric $3\times 3$ doubly stochastic matrix
 with spectrum $1,$ $\lambda,$ $\mu$ if and only if
  $-1\leq \lambda \leq 1,$ $-1\leq \mu\leq 1,$
$\lambda+3\mu+2\geq 0$ and $3\lambda+\mu+2\geq 0.$
\end{theorem}
For $n=4$, a conjecture is given in ~\cite{k:ka} and for $n\geq
4$, only partial solutions are known. Now recall that $\Delta_n^s$
is a convex polytope of dimension $\frac{1}{2}n(n-1)$, whose
vertices were determined in~\cite{k:kat} ( see also~\cite{c:cru})
where it has been proved that if $A$ is a vertex of $\Delta_n^s$,
then
 $A= \frac{1}{2}(P+P^T)$ for some permutation matrix $P$, although
  not every $\frac{1}{2}(P+P^T)$ is a vertex. In addition, $e_{n}$ is always an
eigenvector of any $n\times n$ doubly stochastic matrix $A$ so
that when $A$ is symmetric, then it is written as $A=V_{0}\Lambda
V_{0}^{T}$ for some $ V_{0}$ which is orthogonal and has pattern
$\mathcal{S}$. Therefore for a fixed such $V_{0}$ one can ask what
relation the $\{\lambda _{i}\}$ should satisfy for $V_{0}\Lambda
V_{0}^{T}$ to be symmetric doubly stochastic. Then the region
obtained in this way is a convex region of $ \mathbf{R}^{n}$ and
the largest possible subregion of $\Theta_n^s$ one could obtain in
this fashion, is attained when the columns of the matrix $V_{0}$
are the common set of eigenvectors of a maximum number of vertices
of $\Delta_n^s$ that mutually commute (note that this is the case
in ~\cite{p:per,s:sou}). However this only gives partial solutions
since the classification of all such $V_{0} $ seems to be a
difficult problem. Note that classifying all such $V_{0}$ is
equivalent to the problem of finding the collection $\mathfrak{D}$
of all vectors in $ \mathbf{R}^{n}$ that can serve as the set of
eigenvectors of an $n\times n$  symmetric doubly stochastic
matrix. Since $e_{n}$ is always an eigenvector which is orthogonal
to all other eigenvectors, then the collection
$\mathfrak{D}\backslash \{e_n\}$ is contained in the hyperplane of
$ \mathbf{R}^{n}$ which is orthogonal to $e_{n}.$ Equivalently, if
\ $e_{n},x_{2},...,x_{n}$ are the orthonormal eigenvectors of a
symmetric doubly stochastic matrix $A$ corresponding to the
eigenvalues $1,\lambda _{2},...,\lambda _{n}$\ respectively, then
$A= \frac{1}{n}e_{n}e_{n}^{T}+\lambda
_{2}x_{2}x_{2}^{T}+...+\lambda _{n}x_{n}x_{n}^{T}=$ $
J_{n}+\lambda _{2}x_{2}x_{2}^{T}+...+\lambda _{n}x_{n}x_{n}^{T}.$
Therefore, if $x_{2},...,x_{n}$ are $n-1$ vectors in
$\mathbf{R}^{n}$ such that the sum of the components in each of
them is zero, then giving $1\geq \lambda _{2}\geq ...\geq \lambda
_{n}\geq -1$, one can look at the conditions for which the matrix
$J_{n}+\lambda _{2}x_{2}x_{2}^{T}+...+\lambda _{n}x_{n}x_{n}^{T}$
is nonnegative. Next, we state the earliest result for the (SDIEP)
which is found in~ \cite{p:per}.
\begin{theorem}\cite{p:per} If $1\geq\lambda_{2}\geq ...\geq\lambda_{n}\geq -1$
and
\begin{equation}
\frac{1}{n}+\frac{1}{n(n-1)}\lambda_{2}+\frac{1}{(n-1)(n-2)}
\lambda_{3}+...+\frac{1}{(2)(1)}\lambda_{n} \geq0
\end{equation}
then there is a symmetric doubly-stochastic matrix $D$ such that
$D$ has eigenvalues $1,\lambda_{2},...,\lambda_{n}$.
\end{theorem}

G. Soules~\cite{s:sou} generalized the above result by considering
the following orthogonal pattern ${\mathcal{S}}$ matrix (also
known in the literature as a Soules matrix):
$$
V_{n}=\left(
\begin{array}{ccccccc}
\frac{1}{\sqrt{n}} & \frac{1}{\sqrt{n(n-1)}} & \frac{1}{\sqrt{(n-1)(n-2)}} &
. & . & \frac{1}{\sqrt{6}} & \frac{1}{\sqrt{2}} \\
\frac{1}{\sqrt{n}} & \frac{1}{\sqrt{n(n-1)}} & \frac{1}{\sqrt{(n-1)(n-2)}} &
. & . & \frac{1}{\sqrt{6}} & \frac{-1}{\sqrt{2}} \\
\frac{1}{\sqrt{n}} & \frac{1}{\sqrt{n(n-1)}} & \frac{1}{\sqrt{(n-1)(n-2)}} &
. & . & \frac{-2}{\sqrt{6}} & 0 \\
\frac{1}{\sqrt{n}} & \frac{1}{\sqrt{n(n-1)}} & \frac{1}{\sqrt{(n-1)(n-2)}} &
. & . & 0 & 0 \\
\vdots & \vdots & \vdots & \vdots & \vdots & \vdots &  \\
\frac{1}{\sqrt{n}} & \frac{1}{\sqrt{n(n-1)}} & \frac{1}{\sqrt{(n-1)(n-2)}} &
. & . & 0 & 0 \\
\frac{1}{\sqrt{n}} & \frac{1}{\sqrt{n(n-1)}} & -\frac{n-2}{\sqrt{(n-1)(n-2)}}
& . & . & 0 & 0 \\
\frac{1}{\sqrt{n}} & -\frac{n-1}{\sqrt{n(n-1)}} & 0 & . & . & 0 & 0 \\
&  &  &  &  &  &
\end{array}
\right) .
$$

It is easy to check (see~\cite{s:sou}) that the symmetric matrix
$A=V_{n}\Lambda V_{n}^{T}$ has nonnegative off diagonal entries
while the $i$th diagonal entry of $A$ is given by the convex sum
\[
a_{ii}=\frac{1}{n}+\sum_{k=i+1}^{n}(\frac{1}{(k-1)k})\lambda
_{n-k+2}+(\frac{ i-1}{i})\lambda _{n-i+2},
\]
for $i=1,...,n.$ Moreover the $a_{ii}$ are increasing so the
smallest one is $a_{11}.$ So that if $\Gamma $ is the convex
region defined as the set of all n-tuples $(\lambda
_{1},...,\lambda _{n})$ satisfying: $1\geq \lambda _{2}\geq
...\geq \lambda _{n}\geq -1$ and $1+ \lambda _{2}+ ...+\lambda
_{n}\geq 0$ with $a_{ii}\geq 0 \mbox {  for  }i=1,...,n,$ then
each point in $\Gamma $ is a solution for the (SDIEP). Next,
Soules obtains a larger region than $\Gamma $ by constructing
another matrix $V_{\beta }$ from the matrix $V_{n}$ above as
follows:\newline Given an $s$-long strictly increasing sequence
$\beta $ with values in $ \{1,...,n\}$ where $0<s<n$, let
$\bar{\beta}$ be the $t$-long sequence complementary to $\beta $
$(t=n-s)$. Define the $s$-vector $
u=(\frac{1}{\sqrt{n}},...,\frac{1}{\sqrt{n}})^T\in \mathbf{R}^{s}$
and the $t$-vector
$\bar{u}=(\frac{1}{\sqrt{n}},...,\frac{1}{\sqrt{n}})^T\in
\mathbf{R}^{n-s}$. Now let
 $w=(\sqrt{\frac{n-s}{s}})u$ and $\bar{w}=-(\sqrt{\frac{s}{n-s}})\bar{u}$. Finally,
 define $B_{u}$ as the $s\times (s-1)$ matrix obtained from $V_s$
by deleting the first column and also $B_{\bar{u}}$ is defined
similarly. Next, for any $\beta $ define:
\[
V_{\beta }=\left(
\begin{array}{cccc}
u & w & B_{u} & 0_{1} \\
\bar{u} & \bar{w} & 0_{2} & B_{\bar{u}} \\
&  &  &
\end{array}
\right)
\]
where $0_{1}$ and $0_{2}$ are respectively the $s\times (t-1)$ and
$t\times (s-1)$ zero matrices. Then it is a routine computation to
check (see~\cite{s:sou}) that $V_{\beta }$ is orthogonal and the
symmetric matrix $A(\beta )=V_{\beta }\Lambda V_{\beta }^{T}$ has
nonnegative off diagonal entries. Now, if we let
$V_{\beta}^{\sigma }$ be the matrix obtained from $V_{\beta }$ by
permuting its columns according to a permutation $\sigma $ that
leaves columns $1$ and $2$ fixed and retains the relative order of
the columns in $B_{u}$ and $B_{\bar{u}}$ i.e.
\begin{equation}
\sigma _{j}=j\mbox{ for }j=1,2\mbox{ and }\sigma _{j}<\sigma
_{j+1} \mbox{ for }j\neq s+1,
\end{equation}
then again the symmetric matrix $A(\sigma ,\beta )=V_{\beta
}^{\sigma }\Lambda (V_{\beta }^{\sigma })^{T}$ would have
nonnegative off diagonal entries. Moreover, if $\alpha $ is an
$(s-1)$-long strictly increasing sequence with values in
$\{3,...,n\}$ defined by: $\alpha (j-2)=\sigma _{j}$ for
$j=3,...,s+1$, and $\bar{\alpha}$ is the $(t-1)$-long sequence
complementary to $\alpha $ defined by: $\bar{\alpha}(j-s-1)=\sigma
_{j}$ for $j=s+2,...,n,$ then each diagonal entry of $A(\sigma
,\beta )$ is either equal to $\Sigma_{\alpha}$ or
$\Sigma_{\bar{\alpha}}$ where $\Sigma_{\alpha}$ denote the
following convex sum:
$$
\Sigma_{\alpha}=\frac{1}{n}+\frac{n-s}{ns}\lambda_{2}+\sum
_{k=1}^{s-1}\frac{\lambda_{\alpha(s-k)}}{( k+1)k}.
$$
and $\Sigma_{\bar{\alpha}}$ is obtained from $\Sigma_{\alpha}$ by
replacing $\alpha$ with $\bar{\alpha}.$ Finally, let $n=2m+2$ for
$n$ even and $n=2m+1$ for $n$ odd. Now taking $ \alpha$ to be the
$m$-long sequence $\{3,5,7,...,n-2,n\}$, then it is easy to see
that $\Sigma_{\bar{\alpha}}$ can be obtained from
$\Sigma_{\alpha}$ by replacing $\lambda_{n-2k+2}$ with
$\lambda_{n-2k+1}$ for $k=1,...,m.$ So that $\Sigma_{\alpha}$ is
the smallest and then we have the following corollary:
\begin{corollary}\cite{s:sou} If $1 \geq\lambda_{2}\geq...\geq\lambda_{n}\geq -1$
 and
\begin{equation}
\frac{1}{n}+\frac{n-m-1}{n(m+1)}\lambda_{2}+\sum_{k=1}^{m}\frac{
\lambda_{n-2k+2}}{(k+1)k} \geq 0
\end{equation}
holds, where $n=2m+2$ for $n$ even and $n=2m+1$ for $n$ odd , then
there exists an $n\times n$ symmetric doubly stochastic matrix $D$
such that $D$ has eigenvalues $1,\lambda_{2},...,\lambda_{n}$.
\end{corollary}

\subsection{Remarks on Soules's results.}

We can improve on the above corollary by choosing a `better'
permutation than $\alpha$ above. By that we mean a permutation
that allows us to obtain a larger region of $\Theta_n^{s},$
although this happens on the account of not having a unified
condition for all cases of $n.$ Indeed, we consider the following
4 cases:
\begin{itemize}
\item{$\underline{n=4k+1}:$}\newline Taking $\alpha=\{
3,4,7,8,...,n-2,n-1\}$ and $\bar{\alpha}
=\{5,6,9,10,...,n-4,n-3,n\}$ then we obtain the region (4) defined
by:{\scriptsize
\begin{equation}
\left \{
\begin{array}
[c]{l}
\frac{1}{n}+\frac{n-m-1}{n(m+1)}\lambda_{2}+\frac{\lambda_{3}}{m(m+1)}
+\frac{\lambda_{4}}{(m-1)m}+\frac{\lambda_{7}}
{(m-2)(m-1)}+..+\frac{\lambda_{n-2}}{6}+\frac{\lambda_{n-1}}{2}
\geq0\\
\frac{1}{n}+\frac{n-m}{nm}\lambda_{2}+\frac{\lambda_{5}}{(m-1)m}+
\frac{\lambda_{6}}{(m-2)(m-1)}+\frac{\lambda_{9}}
{(m-3)(m-2)}+..+\frac{\lambda_{n-4}}{12}+\frac
{\lambda_{n-3}}{6 }+\frac{\lambda_{n}}{2} \geq0\\
\end{array}
\right .
\end{equation}}
\item $\underline{n=4k+3}:$ \newline let
$\alpha=\{3,4,7,8,...,n-4,n-3,n\}$ and then $\bar{\alpha}
=\{5,6,9,10,...,n-2,n-1\}$ and we obtain the region (5) defined
by: {\scriptsize
\begin{equation}
\left \{
\begin{array}
[c]{l}
\frac{1}{n}+\frac{n-m-1}{n(m+1)}\lambda_{2}+\frac{\lambda_{3}}{
m(m+1)}+\frac{\lambda_{4}}{(m-1)m}+\frac{\lambda_{7}}
{(m-2)(m-1)}+..+\frac{\lambda_{n-4}}{12}+\frac
{\lambda_{n-3}}{6} +\frac{\lambda_{n}}{2} \geq0 \\

\frac{1}{n}+\frac{n-m}{nm}\lambda_{2}+\frac{\lambda_{5}}{(m-1)m}+
\frac{\lambda_{6}}{(m-2)(m-1)}+\frac{\lambda_{9}}{(m-3)(m-2)}
+...+\frac{\lambda_{n-2}}{6}+\frac
{\lambda_{n-1}}{2 } \geq0 \\
\end{array}
\right .
\end{equation}
} \item $\underline{n=4k+2}:$ \newline Take $\alpha
=\{3,6,7,10,11,....,n-4,n-3,n\}$ and $\bar{\alpha}
=\{4,5,8,9,...,n-2,n-1\}$ and the region (6) obtained is :
{\scriptsize
\begin{equation}
\left \{
\begin{array}
[c]{l}
\frac{1}{n}+\frac{n-m-1}{n(m+1)}\lambda_{2}+\frac{\lambda_{3}}{
m(m+1)}+\frac{\lambda_{6}}{(m-1)m}+\frac{\lambda_{7}}
{(m-2)(m-1)}+...+\frac{\lambda_{n-4}}{12}+\frac{\lambda_{n-3}}{6}+\frac{\lambda_{n}}{2} \geq0 \\
\frac{1}{n}+\frac{n-m-1}{n(m+1)}\lambda_{2}+\frac{\lambda_{4}}{m(m+1)}+
\frac{\lambda_{5}}{(m-1)m}+\frac{\lambda_{8}}{(m-2)(
m-1)}+...+\frac{\lambda_{n-2}}{6}+\frac {\lambda_{n-1}}{2} \geq0 \\
\end{array}
\right .
\end{equation}
} \item $\underline{n=4k+4}$ :\newline Here
$\alpha=\{3,6,7,10,11,...,n-2,n-1\}$ and $\bar{\alpha}
=\{4,5,8,9,...,n-4,n-3,n\}$ and the region obtained (7) is given
by : {\scriptsize
\begin{equation}
\left \{
\begin{array}
[c]{l}
\frac{1}{n}+\frac{n-m-1}{n(m+1)}\lambda_{2}+\frac{\lambda_{3}}{
m(m+1)}+\frac{\lambda_{6}}{(m-1)m}+\frac{\lambda_{7}}{(m-2)(m-1)}+...+\frac{
\lambda_{n-2}}{6}+\frac{\lambda_{n-1}}{2} \geq0 \\
\frac{1}{n}+\frac{n-m-1}{n(m+1)}\lambda_{2}+\frac{\lambda_{4}}{m(m+1)}+
\frac{\lambda_{5}}{(m-1)m}+\frac{\lambda_{8}}{(m-2)(m
-1)}+...+\frac{\lambda_{n-4}}{12}+\frac {\lambda_{n-3}}{6}+\frac{
\lambda_{n}}{2} \geq0
\end{array}
\right .
\end{equation}
}
\end{itemize}

Inequality (3) can be rewritten as :{\scriptsize
\begin{equation}
\frac{1}{n}+\frac{n-m-1}{n(m+1)}\lambda_{2}+\frac{\lambda_{4}}{
m(m+1)}+\frac{\lambda_{6}}{(m-1)m}+\frac{\lambda_{8}}{(m-2)(m-1)}+...+
\frac{\lambda_{n-2}}{6}+\frac{ \lambda_{n}}{2} \geq0
\end{equation}
} for n even {\scriptsize
\begin{equation}
\frac{1}{n}+\frac{n-m-1}{n(m+1)}\lambda_{2}+\frac{\lambda_{3}}{
m(m+1)}+\frac{\lambda_{5}}{(m-1)m}+\frac{\lambda_{7}}{(m-2)(m-1)}+...+
\frac{\lambda_{n-2}}{6}+\frac{ \lambda_{n}}{2} \geq0
\end{equation}
} for n odd.

 Since the $\{ \lambda_i\}$ are in the decreasing order, then
  the left-hand side of (8) is less than
or equal each of the left-hand sides of both (6) and (7). Also,
note that $\frac{n-m}{nm}=\frac{n-m-1}{n(m+1)}+\frac {1}{m(m+1)}$
therefore the left-hand side of (9) is less than or equal each of
the left-hand sides of both (4) and (5). So that the regions (4),
(5), (6) and (7) are bigger. Here is an example that shows that
the region (8) is a proper subset of the region (6). Consider
$(1,1,1,-1/2,-1/2,-1)\in\mathbf{R^{6}}$ which clearly satisfies
(6) and not (8) for $n=6$.
\subsection{ New Methods for Constructing Symmetric Doubly
Stochastic Matrices} Here we describe an algorithm for practical
purposes that yields many new sufficient conditions for the
symmetric doubly stochastic inverse eigenvalue problem. In fact,
this can thought as a generalization of the construction described
above. In addition, this algorithm can also be  used with minor
changes to find solutions for the (DIEP) in the complex case (see
below). Before exploring this algorithm, it helps to think of the
following simple observation which is the building block of the
algorithm.
\begin{observation}
Let $x$ be an eigenvector of an $n\times n$ doubly stochastic
matrix $A$ corresponding to the eigenvalue $\lambda.$ If $0_p$
denote the $p\times 1$ zero vector, then
 $\left[
\begin{array}{c}
0_p \\
x
\end{array}
\right] $ and $\left[
\begin{array}{c}
x \\
0_p
\end{array}
\right] $ are respectively eigenvectors of the $(n+p)\times (n+p)$
doubly stochastic matrices $I_p\oplus A$ and $A\oplus I_p$
corresponding to $\lambda$.
\end{observation}

\begin{algorithm}
  \underline{Setp1}: For low dimension $k=2,3,4,5,6,...$ we consider all
  the vertices of $ \Delta _k^s.$ For a chosen one of these vertices,
  we find its eigenvectors (using Maple for example) as
orthonormal columns vectors
$\frac{1}{\sqrt{k}}e_{k}$,$ x_{2},...,x_{k}.$\\
 \underline{Setp2}:  From the computed eigenvectors in Step1, we form the $k\times k$
orthogonal pattern ${\mathcal{S}}$ matrix
$X_k=(\frac{1}{\sqrt{k}}e_{k}|x_{2}|...|x_{k})$ obtained from
these eigenvectors,  and if
 we let $ \Lambda_k $ be the $k\times k$
diagonal matrix with diagonal entries $ 1,\lambda _{2},...,\lambda
_{k}$ with $1\geq \lambda _{2}\geq ...\geq \lambda _{k}\geq -1,$
and $1+ \lambda _{2}+ ...+ \lambda _{k}\geq 0,$ then we check the
conditions such that $X_k\Lambda_k X_k^{T} \geq 0$ for which by
this construction, we always have solutions.\\
\underline{Setp3}:
Next, for $n>k$ we construct an $n \times n$ matrix $W_{n}$ by
taking the first $n-k+1$ columns of the Soules matrix $V_{n}$ and
the last $(k-1)$ columns $\left(
\begin{array}{c}
x_{2} \\
0 \\
\vdots \\
0
\end{array}
\right) $,...,$\left(
\begin{array}{c}
x_{k} \\
0 \\
\vdots \\
0
\end{array}
\right) $ as the $n$ columns of $W_n$. Then find the region for
which $W_{n}\Lambda W_{n}^T$ is nonnegative.\\
\underline{Setp4}: Here we start the improvement process as
follows. We construct another matrix $W_{\beta}$ which will be the
analogue of the matrix $V_{\beta}$ (defined above) in the
following manner: For $k<s<n$, let $u$ be the $s$-vector
$(\frac{1}{\sqrt{n}},...,\frac{1}{ \sqrt{n}})$, $\bar{u}$ be the
$(n-s)$-vector $(\frac {1}{\sqrt{n}},..., \frac{1}{\sqrt{n}})$,
$w=(\sqrt{\frac{n-s}{s}})u$ and
$\bar{w}=-(\sqrt{\frac{s}{n-s}})\bar{u}$ be defined as in the
previous section. Now define $W_{u}$ be the $(s-1)\times s$ matrix
obtained from the $s\times s$ matrix $W_{s}$ by deleting the first
column and $B_{\bar{u}}$ be the $ (n-s-1)\times (n-s)$ matrix
obtained from the Soules matrix $V_{n-s}$ by deleting the first
column. Then we distinguish between the two cases:
\begin{itemize}
\item
 Case $1$: For $n =k+1,$ we take $W_{n}=W_{k+1}$ and we then find
  the region for which $W_{k+1}\Lambda W_{k+1}^T$ is nonnegative.
\item
 Case $2$: For $n > k+1,$ we let
$$
W_{\beta}=\left(
\begin{array}{cccc}
u & w & 0_{1} & W_{u} \\
\bar{u} & \bar{w} & B_{\bar{u}} & 0_{2} \\
&  &  &
\end{array}
\right)
$$
$0_{1}$ and $0_{2}$ are zero matrices of suitable orders. Then for
any permutation $\alpha$ of the columns of $V_{\beta}$ which
preserves the relative order of columns in $W_{u}$ and
$B_{\bar{u}},$ the conditions for which $W_{\beta}\Lambda
W_{\beta}^T$ is nonnegative give a convex region
$\Gamma_{\alpha}\subset \Theta_n^{s} $ whose all points are
solutions for the (SDIEP). Finally, it is worth mentioning here
that the union of all $\Gamma_{\alpha}$ for all such permutation
$\alpha,$ gives a union of convex subsets in $\Theta_n^{s}$ whose
again all points are solutions for the (SDIEP).\end{itemize}
\underline{Setp5} Repeat the same process for a different chosen
vertex of $ \Delta _k^s.$
\end{algorithm}
\begin{remark} It should be noted that the above algorithm can
also be used by starting with any $k\times k$ doubly stochastic
matrix (i.e. any point of $\Delta_k^s$) that is not necessarily a
vertex of $ \Delta _k^s$ which is the version of Algorithm1 that
we will use for the case of (RDEIP) (see Section 3). However, the
advantage of taking a vertex lies in obtaining a larger subset of
$\Theta_k^{s}$ which in turn by exploiting Step3 and Step4 of
Algorithm1 gives a larger region of $\Theta_n^{s}.$ In addition,
the above algorithm can also be used not just for low dimension,
but for any dimension $k$ for which there exists a $k\times k$
permutation matrix (or any $k\times k$ doubly stochastic matrix)
$P$ such that the eigenvectors of $P$ can be computed.
\end{remark}
\begin{remark} With Algorithm 1, the Soules matrix $V_n$  can be
obtained by choosing the vertex $ P=\left(
\begin{array}{cc}
0 & 1  \\
1 & 0  \\
\end{array}
\right) $  of $\Delta_2^s$ in Step1 since
$(\frac{1}{\sqrt{2}},\frac{1}{\sqrt{2}})^T$ and
$(\frac{1}{\sqrt{2}},-\frac{1}{\sqrt{2}})^T$ are the eigenvectors
of $P.$
\end{remark}
\paragraph{\bf{Example 2}}

We apply Algorithm 1 on the matrix $X=\frac{1}{2} \left(
\begin{array}{cccc}
1 & 1 & 1 & 1 \\
1 & -1 & -1 & 1 \\
1 & -1 & 1 & -1 \\
1 & 1 & -1 & -1 \\
&  &  &
\end{array}
\right) $ which is the orthogonal pattern ${\mathcal{S}}$  matrix
that diagonalizes all the $ 4\times4$ zero-trace symmetric doubly
stochastic matrices (see~\cite{m:mour}). Note that the columns of
$X$ are the orthonormal eigenvectors of the following vertices of
$\Delta_4^s$:
$$p1=\left(
\begin{array}{cccc}
0 & 1 & 0 & 0 \\
1 & 0 & 0 & 0 \\
0 & 0 & 0 & 1 \\
0 & 0 & 1 & 0 \\
&  &  &
\end{array}
\right), \  \  p2=\left(
\begin{array}{cccc}
0 & 0 & 1 & 0 \\
0 & 0 & 0 & 1 \\
1 & 0 & 0 & 0 \\
0 & 1 & 0 & 0 \\
&  &  &
\end{array}
\right), \mbox{ and }  \  p3=\left(
\begin{array}{cccc}
0 & 0 & 0 & 1 \\
0 & 0 & 1 & 0 \\
0 & 1 & 0 & 0 \\
1 & 0 & 0 & 0 \\
&  &  &
\end{array}
\right)
$$
A simple matrix multiplication shows that
 $X
\Lambda_4 X^T=$
$$\frac{1}{2} \left(
\begin{array}{cccc}
1 & 1 & 1 & 1 \\
1 & -1 & -1 & 1 \\
1 & -1 & 1 & -1 \\
1 & 1 & -1 & -1 \\
&  &  &
\end{array}
\right) \left(
\begin{array}{cccc}
1 & 0 & 0 & 0 \\
0 & \lambda_2 & 0 & 0 \\
0 & 0 & \lambda_3& 0 \\
0 & 0 & 0 & \lambda_4 \\
&  &  &
\end{array}
\right) \frac{1}{2} \left(
\begin{array}{cccc}
1 & 1 & 1 & 1 \\
1 & -1 & -1 & 1 \\
1 & -1 & 1 & -1 \\
1 & 1 & -1 & -1 \\
&  &  &
\end{array}
\right)=
$$

$$ \left(
\begin{array}{cccc}
\frac{1}{4}+\frac{\lambda_2}{4}+\frac{\lambda_3}{4}+\frac{\lambda_4}{4}
&
\frac{1}{4}-\frac{\lambda_2}{4}-\frac{\lambda_3}{4}+\frac{\lambda_4}{4}
&
\frac{1}{4}-\frac{\lambda_2}{4}+\frac{\lambda_3}{4}-\frac{\lambda_4}{4}
&
\frac{1}{4}+\frac{\lambda_2}{4}-\frac{\lambda_3}{4}-\frac{\lambda_4}{4} \\
\frac{1}{4}-\frac{\lambda_2}{4}-\frac{\lambda_3}{4}+\frac{\lambda_4}{4}
&
 \frac{1}{4}+\frac{\lambda_2}{4}+\frac{\lambda_3}{4}+\frac{\lambda_4}{4} &
  \frac{1}{4}+\frac{\lambda_2}{4}-\frac{\lambda_3}{4}-\frac{\lambda_4}{4} &
\frac{1}{4}-\frac{\lambda_2}{4}+\frac{\lambda_3}{4}-\frac{\lambda_4}{4}\\
\frac{1}{4}-\frac{\lambda_2}{4}+\frac{\lambda_3}{4}-\frac{\lambda_4}{4}
&
\frac{1}{4}+\frac{\lambda_2}{4}-\frac{\lambda_3}{4}-\frac{\lambda_4}{4}
&
\frac{1}{4}+\frac{\lambda_2}{4}+\frac{\lambda_3}{4}+\frac{\lambda_4}{4}
&
 \frac{1}{4}-\frac{\lambda_2}{4}-\frac{\lambda_3}{4}+\frac{\lambda_4}{4} \\
\frac{1}{4}+\frac{\lambda_2}{4}-\frac{\lambda_3}{4}-\frac{\lambda_4}{4}
&
\frac{1}{4}-\frac{\lambda_2}{4}+\frac{\lambda_3}{4}-\frac{\lambda_4}{4}
&
\frac{1}{4}-\frac{\lambda_2}{4}-\frac{\lambda_3}{4}+\frac{\lambda_4}{4}&
\frac{1}{4}+\frac{\lambda_2}{4}+\frac{\lambda_3}{4}+\frac{\lambda_4}{4} \\
&  &  &
\end{array}
\right).
$$
Since the eigenvalues of $\Lambda_4$ are in the decreasing order
and $trace(\Lambda_4)\geq 0,$ then the only condition for which $X
\Lambda_4 X^T$ is nonnegative is simply given by :
\begin{equation}
\lambda_{1}-\lambda_{2}-\lambda_{3}+\lambda_{4}\geq0.
\end{equation}
Next, we construct the $n\times n$ orthogonal matrix $W_{n}$ which
in this case, is given by:
$$
W_{n}=\left(
\begin{array}{cccccccc}
\frac{1}{\sqrt{n}} & \frac{1}{\sqrt{n(n-1)}} &
\frac{1}{\sqrt{(n-1)(n-2)}} & \ldots & \frac{1}{\sqrt{20}} &
\frac{1}{2} & \frac{1}{2} &
\frac{1}{2} \\
\frac{1}{\sqrt{n}} & \frac{1}{\sqrt{n(n-1)}} &
\frac{1}{\sqrt{(n-1)(n-2)}} & \ldots & \frac{1}{\sqrt{20}} &
\frac{-1}{2} & \frac{-1}{2} &
\frac{1}{2} \\
\frac{1}{\sqrt{n}} & \frac{1}{\sqrt{n(n-1)}} &
\frac{1}{\sqrt{(n-1)(n-2)}} & \ldots & \frac{1}{\sqrt{20}} &
\frac{-1}{2} & \frac{1}{2} &
\frac{-1}{2} \\
\frac{1}{\sqrt{n}} & \frac{1}{\sqrt{n(n-1)}} &
\frac{1}{\sqrt{(n-1)(n-2)}} & \ldots & \frac{1}{\sqrt{20}} &
\frac{1}{2} & \frac{-1}{2} &
\frac{-1}{2} \\
\frac{1}{\sqrt{n}} & \frac{1}{\sqrt{n(n-1)}} & \frac{1}{\sqrt{(n-1)(n-2)}} &
\ldots & \frac{-4}{\sqrt{20}} & 0 & 0 & 0 \\
\frac{1}{\sqrt{n}} & \frac{1}{\sqrt{n(n-1)}} & \frac{1}{\sqrt{(n-1)(n-2)}} &
\ldots & 0 & 0 & 0 & 0 \\
. & . & . & . & . & . & . & . \\
\frac{1}{\sqrt{n}} & \frac{1}{\sqrt{n(n-1)}} & -\frac{n-2}{\sqrt{(n-1)(n-2)}}
& \ldots & 0 & 0 & 0 & 0 \\
\frac{1}{\sqrt{n}} & -\frac{n-1}{\sqrt{n(n-1)}} & 0 & \ldots & 0 & 0 & 0 & 0
\\
&  &  &  &  &  &  &
\end{array}
\right) .
$$
 Then a simple matrix multiplication shows the matrices $A=W_{n}\Lambda W_{n}^{T}$ and $B=V_{n}\Lambda
 V_{n}^{T}$ differ only by the $4\times 4$ principal submatrices
 formed from the first 4 rows and the first 4 columns of $A$ and
 $B$ i.e. $A= \left(
\begin{array}{cc}
X & C  \\
C^T & D
\end{array}
\right)  $ and $B= \left(
\begin{array}{cc}
Y & C  \\
C^T & D
\end{array}
\right)  $  where $X$ and $Y$ are $4\times 4$ matrices. So that
by~\cite{s:sou} all the entries of the symmetric matrix
$A=W_{n}\Lambda W_{n}^{T}$ are nonnegative except for the diagonal
entries and for $a_{12}= a_{21}=a_{34}=a_{43}=$
$$
\frac{1}{n}+\frac{\lambda_{2}}{n(n-1)}+\frac{\lambda_{3}}{
(n-1)(n-2)}+
\frac{\lambda_{4}}{(n-2)(n-3)}+...+\frac{\lambda_{n-3}}{20}-\frac{
\lambda_{n-2}}{4}-\frac{\lambda_{n-1}}{4} +\frac{\lambda_{n}}{4}.
$$
In addition,  the first $4$ diagonal entries are equal to:
$a_{11}=a_{22}=a_{33}=a_{44}=$
$$
\frac{1}{n}+\frac{\lambda_{2}}{n(n-1)}+\frac{\lambda_{3}}{
(n-1)(n-2)}+
\frac{\lambda_{4}}{(n-2)(n-3)}+...+\frac{\lambda_{n-3}}{20}+\frac{
\lambda_{n-2}}{4}+\frac{\lambda_{n-1}}{4}+ \frac{\lambda_{n}}{4}.
$$
and the remaining diagonal entries $a_{55},...,a_{nn}$ are increasing so the
smallest one is

$$
a_{55}=\frac{1}{n}+\frac{\lambda_{2}}{n(n-1)}+\frac{\lambda
_{3}}{(n-1)(n-2)}+...+\frac{\lambda_{n-4}}{30}+
\frac{16}{20}\lambda_{n-3}.
$$

We can rewrite
$$
a_{55}= \frac{1}{n}+\frac{\lambda_{2}}{n(n-1)}+\frac {
\lambda_{3}}{(n-1)(n-2)}+...+\frac{\lambda_{n-3}}{20}+
\frac{\lambda_{n-3}}{4}+
\frac{\lambda_{n-3}}{4}+\frac{\lambda_{n-3}}{4}.
$$
As $1\geq\lambda_{2}\geq...\geq\lambda_{n}\geq -1,$  then clearly
$a_{11}\leq a_{55}$ and therefore the conditions for which $A=
W_{n}^{T}\Lambda W_{n} $ is nonnegative are:
\begin{equation}
\left\{
\begin{array}{ll}
a_{11} & \geq0 \\
a_{12} & \geq0 \\
&
\end{array}
\right.
\end{equation}
Hence we have the following theorem.
\begin{theorem}
Let $1\geq\lambda_{2}\geq ...\geq\lambda_{n}\geq -1$ and
$1+\lambda_{2}+ ...+\lambda_{n}\geq 0.$ If
$$
\left \{
\begin{array}
[c]{l} \frac{1}{n}+\frac{\lambda_{2}}{n(n-1)}+\frac{\lambda_{3}}{
(n-1)(n-2)}+
\frac{\lambda_{4}}{(n-2)(n-3)}+...+\frac{\lambda_{n-3}}{20}+\frac{
\lambda_{n-2}}{4}+\frac{\lambda_{n-1}}{4}+ \frac{\lambda_{n}}{4} \geq0 \\
\frac{1}{n}+\frac{\lambda_{2}}{n(n-1)}+\frac{\lambda_{3}}{
(n-1)(n-2)}+
\frac{\lambda_{4}}{(n-2)(n-3)}+...+\frac{\lambda_{n-3}}{20}-\frac{
\lambda_{n-2}}{4}-\frac{\lambda_{n-1}}{4} +\frac{\lambda_{n}}{4}
\geq0
\end{array}
\right .
$$
then there is an $n\times n$ symmetric doubly stochastic matrix
$D$ such that $D$ has eigenvalues $1,\lambda_{2},...,\lambda_{n}.$
\end{theorem}
To improve upon the results of the above theorem, we consider the
matrix $ W_{\beta}=\left(
\begin{array}{cccc}
u & w & 0_{1} & W_{u} \\
\bar{u} & \bar{w} & B_{\bar{u}} & 0_{2} \\
&  &  &
\end{array}
\right) . $ Now let $\alpha$ be the permutation of the columns of
$ W_{\beta}$ (which of course preserves the relative order of
columns in $W_{u}$ and $B_{\bar{u}}$)  given by:
\begin{itemize}
\item For $n=2m+2$ even and $s=m+1$ and
 $\alpha =\{4,6,8,...,n-6,n-3,n-1,n\}$ then the new improved
conditions become:

{\scriptsize
\begin{equation}
\left \{
\begin{array}
[c]{l}
\frac{1}{n}+\frac{n-m-1}{n(m+1)}\lambda_{2}+\frac{\lambda_{4}}{
m(m+1)}+\frac{\lambda_{6}}{(m-1)m}+
\frac{\lambda_{8}}{(m-2)(m-1)}+...+\frac{
\lambda_{n-6}}{20}+\frac{\lambda_{n-3}}{4}+\frac{\lambda_{n-1}}{4}+
\frac{ \lambda_{n}}{4}\geq0\\
\frac{1}{n}+\frac{n-m-1}{n(m+1)}\lambda_{2}+\frac{\lambda_{4}}{
m(m+1)}+\frac{\lambda_{6}}{(m-1)m}+
\frac{\lambda_{8}}{(m-2)(m-1)}+...+\frac{
\lambda_{n-6}}{20}-\frac{\lambda_{n-3}}{4}-\frac{\lambda_{n-1}}{4}+
\frac{ \lambda_{n}}{4}\geq0 \\
\end{array}
\right .
\end{equation} }

\item  While for $n=2m+1$ odd  and $s=m+1,$
 and $\alpha =\{3,5,7,...,n-6,n-3,n-1,n\}$ the new improved
conditions become: {\scriptsize
\begin{equation}
\left \{
\begin{array}
[c]{l}
\frac{1}{n}+\frac{n-m-1}{n(m+1)}\lambda_{2}+\frac{\lambda_{3}}{
m(m+1)}+\frac{\lambda_{5}}{(m-1)m}+
\frac{\lambda_{7}}{(m-2)(m-1)}+...+\frac{
\lambda_{n-6}}{20}+\frac{\lambda_{n-3}}{4}+\frac{\lambda_{n-1}}{4}+
\frac{ \lambda_{n}}{4}\geq0\\
\frac{1}{n}+\frac{n-m-1}{n(m+1)}\lambda_{2}+\frac{\lambda_{3}}{
m(m+1)}+\frac{\lambda_{5}}{(m-1)m}+
\frac{\lambda_{7}}{(m-2)(m-1)}+...+\frac{
\lambda_{n-6}}{20}-\frac{\lambda_{n-3}}{4}-\frac{\lambda_{n-1}}{4}+
\frac{ \lambda_{n}}{4}\geq0\\
\end{array}
\right .
\end{equation} }
\end{itemize}

\section{The Real Inverse Eigenvalue Problem For Doubly Stochastic Matrices }
First as mentioned earlier, this problem has been considered
in~\cite{s:sot} where the following theorem has been obtained.
\begin{theorem}~\cite{s:sot}\label{th:th2}
Let $\sigma=\{1,\lambda_2,...,\lambda_n\}$ be a set of real
numbers such that $$ 1\geq\lambda_2\geq...\geq\lambda_r\geq 0\geq
\lambda_{r+1}\geq ...\geq\lambda_n.$$ If
$1\geq\lambda_2+n\mbox{max}\{|\lambda_2|,|\lambda_n|\},$ then
there exists an $n\times n$ doubly stochastic matrix with spectrum
$\sigma.$
\end{theorem}
 In addition, the case $n=3$ has been used in a proof of a theorem
concerning the (SDIEP) in~\cite{p:per} in which the following
result holds.
\begin{theorem}
 Let $
X=\left(
\begin{array}{ccc}
p & q & 1-p-q  \\
r & s & 1-r-s \\
1-p-r & 1-q-s & p+q+r+s-1 \\
\end{array}
\right)$ be a doubly stochastic matrix with real eigenvalues $1,$
$\lambda,$ $\mu.$  Then $-1\leq \lambda \leq 1,$ $-1\leq \mu\leq
1,$ $\lambda+3\mu+2\geq 0$ and $3\lambda+\mu+2\geq 0.$
\end{theorem}
\begin{proof}  The first two inequalities are obtained from the
Perron-Frobenius theorem. For the other two inequalities, let
$x=\lambda+3\mu+2$ and $y=3\lambda+\mu+2.$ We prove that $x$ and
$y$ are nonnegative by showing that their sum $x+y$ and their
product $xy$ are nonnegative. For,
$x+y=4\lambda+4\mu+4=4trace(X)\geq 0.$ Now
$xy=3(\lambda+\mu+1)^2+2(\lambda+\mu+1)+4\lambda\mu-1=3[trace(X)]^2+2trace(X)+4determinant(X)-1=3(q-r)^2
+12(p+s)(p+q+r+s-1)+12ps.$  As the entries of $X$ are nonnegative,
therefore $xy\geq 0,$  and the proof is complete.
\end{proof}\\
Combining the above theorem with Theorem~\ref{th:th1},
 we have the following conclusion.
\begin{corollary} The (RDIEP) and the (SDIEP) are equivalent for
the case $n=3$.
\end{corollary}

Now returning to the procedures described at the beginning of
Section2 and starting out with a particular $n\times n$ pattern
$\mathcal{S}$ matrix $V$ which is not orthogonal,
 we obtain the following theorem that solves the (RDIEP) for some restricted cases.

\begin{theorem}
If $1\geq \lambda _{2}\geq ...\geq \lambda _{n}\geq -1$ and
\begin{equation}
\left\{
\begin{array}{c}
1-(n-1)\lambda _{2}+\lambda _{3}+...+\lambda _{n}\geq 0 \\
1+(n-1)\lambda _{n}\geq 0
\end{array}
\right.
\end{equation}
then there is an $n \times n$  nonsymmetric doubly stochastic
matrix $D$ with spectrum $1,\lambda _{2},...,\lambda _{n}.$
\end{theorem}

\begin{proof} Clearly the  matrix
\[
V=\left(
\begin{array}{cccccccc}
1 & 1 & 1 & . & . &.& 1 & 1 \\
1 & -1 & 0 & . & . &.& 0 & 0 \\
1 & 0 & -1 & 0 & . &.& 0 & 0 \\
1 & 0 & 0 & -1 & 0 &.& 0 & 0 \\
\vdots  & \vdots  & \vdots  & \vdots  & \vdots  & \vdots  & \vdots  & \\
1 & 0 & 0 &.& 0& -1 & 0 & 0 \\
1 & 0 & 0 & . & .&.0& -1 & 0 \\
1 & 0 & 0 & . &.& . & 0 & -1
\end{array}
\right)
\] has pattern $\mathcal{S}$ and its inverse $V^{-1}$ is given by:

$$
V^{-1}=\left(
\begin{array}{ccccccc}
\frac{1}{n} & \frac{1}{n}& \frac{1}{n}& \frac{1}{n} & \frac{1}{n}& \frac{1}{n}&\frac{1}{n} \\
\frac{1}{n}& -\frac{(n-1)}{n} & \frac{1}{n}& . & . & .& \frac{1}{n} \\
\frac{1}{n} & \frac{1}{n}& -\frac{(n-1)}{n} & \frac{1}{n} & . & . & \frac{1}{n} \\
\frac{1}{n}& \frac{1}{n} & \frac{1}{n}& -\frac{(n-1)}{n} &\frac{1}{n}& . &  \frac{1}{n} \\
\vdots  & \vdots  & \vdots  & \vdots  & \vdots  & \vdots  & \vdots\\
\frac{1}{n}& \frac{1}{n} & .& \frac{1}{n}& -\frac{(n-1)}{n}& \frac{1}{n}& \frac{1}{n} \\
\frac{1}{n}& . & . & . & \frac{1}{n} & -\frac{(n-1)}{n} & \frac{1}{n} \\
\frac{1}{n}& \frac{1}{n} & \frac{1}{n}& \frac{1}{n} & \frac{1}{n}&
\frac{1}{n} & -\frac{(n-1)}{n}
\end{array}
\right) . $$ Now the entries of the matrix $A=(a_{ij})=V\Lambda
V^{-1}$ satisfy the following relations:

\begin{equation}
\left\{
\begin{array}
[c]{l}
a_{11}=\frac{1}{n}(trace(\Lambda))\\
a_{ii}=\frac{1}{n}(1+(n-1)\lambda_i) \mbox{ for } i=2,...,n \\
a_{i1}=\frac{1}{n}(1+\lambda_2 +...+\lambda_{i-1}-(n-1)\lambda_i+\lambda_{i+1}+...+\lambda_n)\\
a_{ij}=\frac{1}{n}(1-\lambda_j) \mbox { for }  j > 1  \mbox {  and
 }  j\neq i
\end{array}
\right.
\end{equation}
Note that $a_{11}$ and $a_{ij}=\frac{1}{n}(1-\lambda_j) \mbox {
for }  j > 1  \mbox {  and
 }  j\neq i $ are nonnegative since the diagonal entries of
 $\Lambda$ are in the decreasing order. In addition, for $i\neq
 1,$ the entries $a_{i1}$ are increasing so the smallest one is
 $a_{21}$ and $a_{ii}$ are decreasing so the smallest is $a_{nn}.$
 Therefore the matrix $A$ is nonnegative if and only if $a_{21}\geq 0$
 and $a_{nn}\geq 0.$ Finally $A$ is doubly stochastic since $V$
 has the pattern $\mathcal{S},$ and the proof is complete.
\end{proof}

It is should be noted here that the conditions of the above
theorem differ than those of Theorem \ref{th:th2} as the point
$(1,1/2,1/4)$ clearly satisfies the conditions of the above
theorem but obviously does not satisfy the conditions of Theorem
\ref{th:th2} for the case $n=3.$ In addition, the conditions of
the above theorem are also sufficient for the existence of an
$n\times n$ symmetric doubly stochastic matrix $D$ with the
spectrum $(1,\lambda_2,...,\lambda_n).$ To see this, it suffices
to look at the following result which is Corollary 7
in~\cite{h:hw}.
\begin{theorem}~\cite{h:hw} If $\lambda_2,...,\lambda_n\in[-1/(n-1),1]$, then
there exists an $n\times n$ symmetric doubly stochastic matrix $D$
with spectrum $(1,\lambda_2,...,\lambda_n).$
\end{theorem}

 Here it is worth mentioning that at this stage and with
  extensive numerical computations, we are not
 able to find a list of $n$ $(\geq 4)$ real numbers which shows the two
 problems (RDIEP) and (SDIEP) are different. In conclusion, the
 question whether (RDIEP) and (SDIEP) are equivalent or not, for
 $n\geq 4$ remains an open problem.

We conclude this section by using the version of Algorithm 1
mentioned in Remark 2.5. More explicitly, we choose a nonsymmetric
doubly stochastic matrix with real eigenvalues such as the
$3\times 3$ matrix $B$ of Example 1. Using Maple, the eigenvectors
of $B$ are given by $(1,1,1)^T,$  $(1,-2,1)^T$ and $(1,1,-2)^T.$
Now the corresponding pattern $\mathcal{S}$ matrix is given by
$P=\left(
\begin{array}{ccc}
1& 1 & 1  \\
1 & -2 & 1  \\
1 & 1 & -2 \\
\end{array}
\right). $ Then its inverse is given by $P^{-1}=\left(
\begin{array}{ccc}
1/3& 1/3 & 1/3 \\
1/3 & -1/3 & 0 \\
1/3 & 0 & -1/3 \\
\end{array}
\right) $ and then $P\Lambda_3 P^{-1}=\left(
\begin{array}{ccc}
\frac{1+\lambda_2+\lambda_3}{3}& \frac{1-\lambda_2}{3} & \frac{1-\lambda_3}{3}  \\
\frac{1-2\lambda_2+\lambda_3}{3} & \frac{1+2\lambda_2}{3} & \frac{1-\lambda_3}{3}  \\
\frac{1+\lambda_2-2\lambda_3}{3} & \frac{1-\lambda_2}{3} & \frac{1+2\lambda_3}{3} \\
\end{array}
\right). $ Then the conditions for which $P\Lambda_3 P^{-1}$ is
nonnegative is given by Theorem 3.4 for the case $n=3$ with no
surprise as $P$ is a multiple of $V^{-1}$ which is used in the
proof of Theorem 3.4. Following Step4 of Algorithm 1, we construct
the matrix $W_n$ which is given by:
$$ W_{n}=\left(
\begin{array}{cccccccc}
\frac{1}{\sqrt{n}} & \frac{1}{\sqrt{n(n-1)}} &
\frac{1}{\sqrt{(n-1)(n-2)}} &
. & . & \frac{1}{\sqrt{12}} & 1 & 1\\
\frac{1}{\sqrt{n}} & \frac{1}{\sqrt{n(n-1)}} &
\frac{1}{\sqrt{(n-1)(n-2)}} &
. & . & \frac{1}{\sqrt{12}} & -2 & 1 \\
\frac{1}{\sqrt{n}} & \frac{1}{\sqrt{n(n-1)}} &
\frac{1}{\sqrt{(n-1)(n-2)}} &
. & . & \frac{1}{\sqrt{12}} & 1 & -2 \\
\frac{1}{\sqrt{n}} & \frac{1}{\sqrt{n(n-1)}} &
\frac{1}{\sqrt{(n-1)(n-2)}} &
. & . & \frac{-3}{\sqrt{12}}& 0 & 0 \\
\vdots & \vdots & \vdots & \vdots & \vdots & \vdots & \vdots &\vdots \\
\frac{1}{\sqrt{n}} & \frac{1}{\sqrt{n(n-1)}} &
\frac{1}{\sqrt{(n-1)(n-2)}} &
. & . & 0 & 0 & 0\\
\frac{1}{\sqrt{n}} & \frac{1}{\sqrt{n(n-1)}} &
-\frac{n-2}{\sqrt{(n-1)(n-2)}}
& . & . & 0 & 0 & 0 \\
\frac{1}{\sqrt{n}} & -\frac{n-1}{\sqrt{n(n-1)}} & 0 & . & . & 0 & 0 & 0 \\
\end{array}
\right).
$$
A virtually identical description as in the symmetric case shows
that $A=W_n\Lambda W_n^{-1}= \left(
\begin{array}{cc}
X & C  \\
C^T & D
\end{array}
\right)  $ and $B=V_n\Lambda V_n^{-1}= \left(
\begin{array}{cc}
Y & C  \\
C^T & D
\end{array}
\right)  $  where $X$ and $Y$ are $3\times 3$ matrices. Thus a
simple check shows that the conditions for which $A=W_n\Lambda
W_n^{-1}$ is nonnegative give the following theorem.
\begin{theorem}
Let $1\geq \lambda _{2}\geq ...\geq \lambda _{n}\geq -1,$ such
that $ 1+ \lambda _{2}+ ...+ \lambda _{n}\geq 0 .$ If
$$
\left \{
\begin{array}[c]{l}
 \frac{1}{n}+\frac{\lambda_{2}}{n(n-1)}+\frac{\lambda_{3}}{
(n-1)(n-2)}+\frac{\lambda_{4}}{(n-2)(n-3)}+...+\frac{\lambda_{n-2}}{(4)(3)}-\frac{2\lambda_{n-1}}{3}+\frac{\lambda_{n}}{3} \geq0 \\
\frac{1}{n}+\frac{\lambda_{2}}{n(n-1)}+\frac{\lambda_{3}}{
(n-1)(n-2)}+\frac{\lambda_{4}}{(n-2)(n-3)}+...+\frac{\lambda_{n-2}}{(4)(3)}+\frac{2\lambda_{n}}{3} \geq0 \\
\end{array}
\right .
$$
then $(1,\lambda_2, ..., \lambda_n)$ is the spectrum of an
$n\times n$ nonsymmetric doubly stochastic matrix.
\end{theorem}
Next, we use the improvement process of Step4 to obtain new
results. This can be illustrated by the following example for the
case $n=6.$
\paragraph{\bf{Example 3}} we apply the improvement process
described in Step4 of Algorithm 1  to obtain the matrix
$$ W_\beta=\left(
\begin{array}{ccccccc}
\frac{1}{\sqrt{6}} & \frac{1}{\sqrt{6}} &
0 & 0 & 1 & 1 \\
\frac{1}{\sqrt{6}} & \frac{1}{\sqrt{6}}  &
0 & 0 & -2 & 1 \\
\frac{1}{\sqrt{6}} & \frac{1}{\sqrt{6}} &
0 &  0 & 1 & -2 \\
\frac{1}{\sqrt{6}} & -\frac{1}{\sqrt{6}} &
 1 & 1 & 0 & 0\\
 \frac{1}{\sqrt{6}} & -\frac{1}{\sqrt{6}}  &
 -2 & 1 & 0 & 0 \\
 \frac{1}{\sqrt{6}} & -\frac{1}{\sqrt{6}} &
 1 & -2 & 0 & 0\\
\end{array}
\right) .$$  Now an inspection shows that $W_\beta\Lambda_6
W_\beta^{-1}= \left(
\begin{array}{cc}
B & C  \\
C & D
\end{array}
\right)  $ where
$$B=   \left(
\begin{array}{ccc}
\frac{1}{6}+\frac{\lambda_2}{6}+\frac{\lambda_5}{3}+\frac{\lambda_6}{3}&
\frac{1}{6}+\frac{\lambda_2}{6}-\frac{\lambda_5}{3}&
\frac{1}{6}+\frac{\lambda_2}{6}-\frac{\lambda_6}{3}  \\
 \frac{1}{6}+\frac{\lambda_2}{6}-\frac{2\lambda_5}{3}+\frac{\lambda_6}{3} & \frac{1}{6}+\frac{\lambda_2}{6}+\frac{2\lambda_5}{3} &
\frac{1}{6}+\frac{\lambda_2}{6}-\frac{\lambda_6}{3}  \\
\frac{1}{6}+\frac{\lambda_2}{6}+\frac{\lambda_5}{3}-\frac{2\lambda_6}{3}
& \frac{1}{6}+\frac{\lambda_2}{6}-\frac{\lambda_5}{3} &
\frac{1}{6}+\frac{\lambda_2}{6}+\frac{2\lambda_6}{3}
\end{array}
\right)  ,   C=   \left(
\begin{array}{ccc}
  \frac{1}{6}-\frac{\lambda_2}{6} & \frac{1}{6}-\frac{\lambda_2}{6} & \frac{1}{6}-\frac{\lambda_2}{6} \\
  \frac{1}{6}-\frac{\lambda_2}{6} & \frac{1}{6}-\frac{\lambda_2}{6} & \frac{1}{6}-\frac{\lambda_2}{6} \\
\frac{1}{6}-\frac{\lambda_2}{6} & \frac{1}{6}-\frac{\lambda_2}{6}
& \frac{1}{6}-\frac{\lambda_2}{6}
\end{array}
\right)    $$
 and $D= \left(
\begin{array}{ccc}
\frac{1}{6}+\frac{\lambda_2}{6}+\frac{\lambda_3}{3}+\frac{\lambda_4}{3}&
\frac{1}{6}+\frac{\lambda_2}{6}-\frac{\lambda_3}{3}&
\frac{1}{6}+\frac{\lambda_2}{6}-\frac{\lambda_4}{3}  \\
 \frac{1}{6}+\frac{\lambda_2}{6}-\frac{2\lambda_3}{3}+\frac{\lambda_4}{3} & \frac{1}{6}+\frac{\lambda_2}{6}+\frac{2\lambda_3}{3} &
\frac{1}{6}+\frac{\lambda_2}{6}-\frac{\lambda_4}{3}  \\
\frac{1}{6}+\frac{\lambda_2}{6}+\frac{\lambda_3}{3}-\frac{2\lambda_4}{3}
& \frac{1}{6}+\frac{\lambda_2}{6}-\frac{\lambda_3}{3} &
\frac{1}{6}+\frac{\lambda_2}{6}+\frac{2\lambda_4}{3}
\end{array}
\right)  .$
 Note that for $1\geq \lambda _{2}\geq ...\geq \lambda _{6}\geq -1,$ we have $\frac{1}{6}+\frac{\lambda_2}{6}+\frac{2\lambda_4}{3}\geq
 \frac{1}{6}+\frac{\lambda_2}{6}+\frac{2\lambda_6}{3}.$
    Thus the conditions for which the matrix $W_\beta\Lambda_6
W_\beta^{-1}$ is nonnegative result in the following conclusion.

\begin{theorem}
Let $1\geq \lambda _{2}\geq ...\geq \lambda _{6}\geq -1,$ such
that $ 1+ \lambda _{2}+ ...+ \lambda _{6} .$ If
$$
\left \{
\begin{array}[c]{l}
\frac{1}{6}+\frac{\lambda_2}{6}+\frac{2\lambda_6}{3} \geq 0\\
 \frac{1}{6}+\frac{\lambda_2}{6}-\frac{2\lambda_3}{3}+\frac{\lambda_4}{3} \geq0 \\
\frac{1}{6}+\frac{\lambda_2}{6}-\frac{2\lambda_5}{3}+\frac{\lambda_6}{3}\geq0 \\
\end{array}
\right .
$$
then $(1,\lambda_2, ..., \lambda_6)$ is the spectrum of a $6\times
6$ nonsymmetric doubly stochastic matrix
\end{theorem}

\section{Constructing doubly stochastic matrices with complex spectrum}

We start the section by again mentioning that only the case $n=3$
is completely solved in~\cite{p:per} where the following theorem
has been proved.
\begin{theorem}~\cite{p:per}
Let $z$ be a complex number with nonzero imaginary part and
$\bar{z}$ be its complex conjugate. Then $(1,z,\bar{z})$ is the
spectrum of a $3\times 3$ doubly stochastic matrix if and only if
$z$ is in the convex hull of the three cubic roots of unity
\end{theorem}

 To take advantage of Algorithm 1 for this case, we recall the
famous Birkoff's theorem that states that $\Delta_n$ is a convex
polytope of dimension $(n-1)^2$ where its vertices are the
$n\times n$ permutation matrices. Next, we describe how to
manipulate Algorithm 1 to obtain solutions for (DIEP) with complex
spectrum but with minor changes namely  in Step1 we start out with
a nonsymmetric vertex of $\Delta_k.$ In fact, we can begin
Algorithm 1 with any $k\times k$ doubly stochastic matrix $X$ with
complex spectrum but in this case the orthonormalization process
in Step1 should be dropped if the columns of $X$ are not
orthogonal. We illustrate this idea by considering the following
example.

\paragraph{\bf{Example 4}} The vertex $P= \left(
\begin{array}{ccc}
0& 1 & 0  \\
0 & 0 & 1  \\
1 & 0 & 0 \\
\end{array}
\right) $  of $\Delta_3$  has the following three eigenvectors
$\frac{1}{\sqrt{3}}e_3,$ $x_2=\frac{1}{\sqrt{3}}(1,w^2,w)^T$ and
$x_3=\frac{1}{\sqrt{3}}(1,w,w^2)^T$ where $w=-1/2+I\sqrt{3}/2$ is
the primitive cubic root of unity. Then the $3\times 3$ complex
pattern $\mathcal{S}$ matrix obtained from these eigenvectors is
given by $X= \frac{1}{\sqrt{3}}\left(
\begin{array}{ccc}
1& 1 & 1  \\
1 & w^2 & w  \\
1 & w & w^2 \\
\end{array}
\right) $ and its inverse is $X^{-1}=\frac{1}{\sqrt{3}}\left(
\begin{array}{ccc}
1& 1 & 1  \\
1 & w & w^2  \\
1 & w^2 & w \\
\end{array}
\right) .$ If we let $\Pi_3=\left(
\begin{array}{ccc}
1& 0 & 0  \\
0 & a+Ib & 0  \\
 0 & 0 & a-Ib \\
\end{array}
\right) $ with $trace(\Pi_3)=1+2a\geq 0$ and $|a+Ib|\leq 1,$ then
a simple check shows that
$$X\Pi_3 X^{-1}=\frac{1}{3}\left(
\begin{array}{ccc}
1+2a& 1 -a-b\sqrt{3}& 1 -a+b\sqrt{3} \\
1 -a+b\sqrt{3} & 1+2a & 1 -a-b\sqrt{3}  \\
1 -a-b\sqrt{3} & 1 -a+b\sqrt{3} & 1+2a \\
\end{array}
\right) .$$  Therefore the conditions for which $X\Pi_3 X^{-1}$ is
nonnegative are simply giving by:
\begin{equation}
\left\{
\begin{array}
[c]{l}
1-a +b\sqrt{3} \geq 0\\
1-a -b\sqrt{3}\geq 0
\end{array}
\right.
\end{equation}

The next step of our algorithm is to form the following $n\times
n$ complex pattern $\mathcal{S}$ matrix
$$
W_{n}=\left(
\begin{array}{cccccccc}
\frac{1}{\sqrt{n}} & \frac{1}{\sqrt{n(n-1)}} &
\frac{1}{\sqrt{(n-1)(n-2)}} &
. & . & \frac{1}{\sqrt{12}} & \frac{1}{\sqrt{3}} & \frac{1}{\sqrt{3}}\\
\frac{1}{\sqrt{n}} & \frac{1}{\sqrt{n(n-1)}} &
\frac{1}{\sqrt{(n-1)(n-2)}} &
. & . & \frac{1}{\sqrt{12}} & \frac{w^2}{\sqrt{3}}& \frac{w}{\sqrt{3}} \\
\frac{1}{\sqrt{n}} & \frac{1}{\sqrt{n(n-1)}} &
\frac{1}{\sqrt{(n-1)(n-2)}} &
. & . & \frac{1}{\sqrt{12}} & \frac{w}{\sqrt{3}}& \frac{w^2}{\sqrt{3}} \\
\frac{1}{\sqrt{n}} & \frac{1}{\sqrt{n(n-1)}} &
\frac{1}{\sqrt{(n-1)(n-2)}} &
. & . & \frac{-3}{\sqrt{12}}& 0 & 0 \\
\vdots & \vdots & \vdots & \vdots & \vdots & \vdots & \vdots &\vdots \\
\frac{1}{\sqrt{n}} & \frac{1}{\sqrt{n(n-1)}} &
\frac{1}{\sqrt{(n-1)(n-2)}} &
. & . & 0 & 0 & 0\\
\frac{1}{\sqrt{n}} & \frac{1}{\sqrt{n(n-1)}} &
-\frac{n-2}{\sqrt{(n-1)(n-2)}}
& . & . & 0 & 0 & 0 \\
\frac{1}{\sqrt{n}} & -\frac{n-1}{\sqrt{n(n-1)}} & 0 & . & . & 0 & 0 & 0 \\
\end{array}
\right)
$$
Now if we let $\Pi_n$ be the diagonal matrix
$1\oplus\lambda_2\oplus ... \oplus \lambda_{n-2}\oplus a+Ib\oplus
a-Ib,$ such that $|a+Ib|\leq 1$ and with $1\geq \lambda _{2}\geq
...\geq \lambda _{n-2}\geq -1$ and $ trace(\Pi_n)= 1+ \lambda
_{2}+ ...+ \lambda _{n-2}+2a\geq 0.$ Then an inspection shows that
$W_n\Pi_n W_n^{-1}=\left(
\begin{array}{cc}
A& B   \\
B^T & C
\end{array}
\right) $ where $A$ is a $3\times 3$ circulant matrix whose first
row $(a_{11},a_{12},a_{13})$ is given by:
$$
\left \{
\begin{array}[c]{l}
a_{11}= \frac{1}{n}+\frac{\lambda_{2}}{n(n-1)}+\frac{\lambda_{3}}{
(n-1)(n-2)}+\frac{\lambda_{4}}{(n-2)(n-3)}+...+\frac{\lambda_{n-2}}{(4)(3)}+\frac{2a}{3} \\
a_{12}=\frac{1}{n}+\frac{\lambda_{2}}{n(n-1)}+\frac{\lambda_{3}}{
(n-1)(n-2)}+\frac{\lambda_{4}}{(n-2)(n-3)}+...+\frac{\lambda_{n-2}}{(4)(3)}-\frac{a}{3}-\frac{b\sqrt{3}}{3} \\
a_{13}=\frac{1}{n}+\frac{\lambda_{2}}{n(n-1)}+\frac{\lambda_{3}}{
(n-1)(n-2)}+\frac{\lambda_{4}}{(n-2)(n-3)}+...+\frac{\lambda_{n-2}}{(4)(3)}-\frac{a}{3}+\frac{b\sqrt{3}}{3}\\
\end{array}
\right .
$$ and $B$ is nonnegative and $C$ has nonnegative off-diagonal
entries and its diagonal entries are increasing so that the
smallest is
 $$c_{11}=\frac{1}{n}+\frac{\lambda_{2}}{n(n-1)}+\frac{\lambda_{3}}{
(n-1)(n-2)}+\frac{\lambda_{4}}{(n-2)(n-3)}+...+\frac{\lambda_{n-3}}{(5)(4)}+\frac{3\lambda_{n-2}}{4}.$$
 Thus the conditions for which $W_n\Pi_n W_n^{-1}$
is nonnegative gives the following theorem.
\begin{theorem}
Let $1\geq \lambda _{2}\geq ...\geq \lambda _{n-2}\geq -1,$ such
that $ 1+ \lambda _{2}+ ...+ \lambda _{n-2}+2a\geq 0 $ and
$|a+Ib|\leq 1.$ If
$$
\left \{
\begin{array}[c]{l}
 \frac{1}{n}+\frac{\lambda_{2}}{n(n-1)}+\frac{\lambda_{3}}{
(n-1)(n-2)}+\frac{\lambda_{4}}{(n-2)(n-3)}+...+\frac{\lambda_{n-2}}{(4)(3)}+\frac{2a}{3} \geq0 \\
\frac{1}{n}+\frac{\lambda_{2}}{n(n-1)}+\frac{\lambda_{3}}{
(n-1)(n-2)}+\frac{\lambda_{4}}{(n-2)(n-3)}+...+\frac{\lambda_{n-2}}{(4)(3)}-\frac{a}{3}-\frac{b\sqrt{3}}{3} \geq0 \\
\frac{1}{n}+\frac{\lambda_{2}}{n(n-1)}+\frac{\lambda_{3}}{
(n-1)(n-2)}+\frac{\lambda_{4}}{(n-2)(n-3)}+...+\frac{\lambda_{n-2}}{(4)(3)}-\frac{a}{3}+\frac{b\sqrt{3}}{3}
\geq0 \\
\frac{1}{n}+\frac{\lambda_{2}}{n(n-1)}+\frac{\lambda_{3}}{
(n-1)(n-2)}+\frac{\lambda_{4}}{(n-2)(n-3)}+...+\frac{\lambda_{n-3}}{(5)(4)}+\frac{3\lambda_{n-2}}{4}\geq0
\end{array}
\right .$$
then $(1,\lambda_2, ..., \lambda_{n-2}, a+Ib, a-Ib)$ is
the spectrum of an $n\times n$ doubly stochastic matrix
\end{theorem}
Note that the above theorem deals with the case of at most two
nonreal eigenvalues namely $a+Ib$ and $a-Ib$ in the list
$\{1,\lambda_2, ..., \lambda_{n-2}, a+Ib, a-Ib\}.$  Now, we can
use the improvement process of Step4 to obtain conditions on a
list with more than just two complex eigenvalues. This can be
illustrated by the following example again for the case $n=6.$
\paragraph{\bf{Example 5}} we apply the improvement process
described in Step4 of our algorithm to obtain the following matrix
$$ W_\beta=\left(
\begin{array}{ccccccc}
\frac{1}{\sqrt{6}} & \frac{1}{\sqrt{6}} &
0 & 0 & \frac{1}{\sqrt{3}} & \frac{1}{\sqrt{3}}\\
\frac{1}{\sqrt{6}} & \frac{1}{\sqrt{6}}  &
0 & 0 & \frac{w^2}{\sqrt{3}}& \frac{w}{\sqrt{3}} \\
\frac{1}{\sqrt{6}} & \frac{1}{\sqrt{6}} &
0 &  0 & \frac{w}{\sqrt{3}}& \frac{w^2}{\sqrt{3}} \\
\frac{1}{\sqrt{6}} & -\frac{1}{\sqrt{6}} &
 \frac{1}{\sqrt{3}} & \frac{1}{\sqrt{3}}& 0 & 0\\
 \frac{1}{\sqrt{6}} & -\frac{1}{\sqrt{6}}  &
 \frac{w^2}{\sqrt{3}}& \frac{w}{\sqrt{3}} & 0 & 0 \\
 \frac{1}{\sqrt{6}} & -\frac{1}{\sqrt{6}} &
 \frac{w}{\sqrt{3}}& \frac{w^2}{\sqrt{3}} & 0 & 0\\
\end{array}
\right) $$

If we let $\alpha$ be the diagonal matrix $1\oplus f\oplus
a+Ib\oplus a-Ib\oplus c+Id\oplus c-Id,$ such that $|a+Ib|\leq 1,$
  $|c+Id|\leq 1,$  $-1\leq f\leq 1$ and $ trace(\alpha)= 1+ f +2a+2c\geq 0.$
  Then an inspection shows that the matrix
$W_\beta\alpha W_\beta^{-1}$ has the form $W_\beta\alpha
W_\beta^{-1}= \left(
\begin{array}{cc}
B & C  \\
C & D
\end{array}
\right)  $ where $$ B=   \left(
\begin{array}{ccc}
\frac{1}{6}+\frac{f}{6}+\frac{2c}{3} &
\frac{1}{6}+\frac{f}{6}-\frac{c}{3}-\frac{d\sqrt{3}}{3} &
\frac{1}{6}+\frac{f}{6}-\frac{c}{3}+\frac{d\sqrt{3}}{3}  \\
 \frac{1}{6}+\frac{f}{6}-\frac{c}{3}+\frac{d\sqrt{3}}{3} & \frac{1}{6}+\frac{f}{6}+\frac{2c}{3} &
\frac{1}{6}+\frac{f}{6}-\frac{c}{3}-\frac{d\sqrt{3}}{3}  \\
\frac{1}{6}+\frac{f}{6}-\frac{c}{3}-\frac{d\sqrt{3}}{3} &
\frac{1}{6}+\frac{f}{6}-\frac{c}{3}+\frac{d\sqrt{3}}{3} &
\frac{1}{6}+\frac{f}{6}+\frac{2c}{3}
\end{array}
\right) ,   C=   \left(
\begin{array}{ccc}
  \frac{1}{6}-\frac{f}{6} & \frac{1}{6}-\frac{f}{6} & \frac{1}{6}-\frac{f}{6} \\
  \frac{1}{6}-\frac{f}{6} & \frac{1}{6}-\frac{f}{6} & \frac{1}{6}-\frac{f}{6} \\
\frac{1}{6}-\frac{f}{6} & \frac{1}{6}-\frac{f}{6} &
\frac{1}{6}-\frac{f}{6}
\end{array}
\right)     \mbox{ and } $$
 $D=   \left(
\begin{array}{ccc}
  \frac{1}{6}+\frac{f}{6}+\frac{2a}{3} &
 \frac{1}{6}+\frac{f}{6}-\frac{a}{3}-\frac{b\sqrt{3}}{3} & \frac{1}{6}+\frac{f}{6}-\frac{a}{3}+\frac{b\sqrt{3}}{3} \\
\frac{1}{6}+\frac{f}{6}-\frac{a}{3}+\frac{b\sqrt{3}}{3} &
\frac{1}{6}+\frac{f}{6}+\frac{2a}{3}&
\frac{1}{6}+\frac{f}{6}-\frac{a}{3}-\frac{b\sqrt{3}}{3} \\
\frac{1}{6}+\frac{f}{6}-\frac{a}{3}-\frac{b\sqrt{3}}{3} &
\frac{1}{6}+\frac{f}{6}-\frac{a}{3}+\frac{b\sqrt{3}}{3} &
\frac{1}{6}+\frac{f}{6}+\frac{2a}{3}
\end{array}
\right) . $ Now the conditions for which the matrix $W_\beta\alpha
W_\beta^{-1}$ is nonnegative result in the following conclusion.
\begin{theorem}
Let $1, f, a+Ib, a-Ib, c+Id, c-Id$ be complex numbers such that
$|a+Ib|\leq 1,$ $|c+Id|\leq 1,$  $-1\leq f\leq 1$ and $ 1+ f
+2a+2c\geq 0.$   If
$$
\left \{
\begin{array}[c]{l}
\frac{1}{6}+\frac{f}{6}+\frac{2a}{3}\geq 0\\
 \frac{1}{6}+\frac{f}{6}+\frac{2c}{3}\geq 0\\
 \frac{1}{6}+\frac{f}{6}-\frac{a}{3}+\frac{b\sqrt{3}}{3}\geq 0\\
 \frac{1}{6}+\frac{f}{6}-\frac{a}{3}-\frac{b\sqrt{3}}{3}\geq 0\\
 \frac{1}{6}+\frac{f}{6}-\frac{c}{3}+\frac{d\sqrt{3}}{3}\geq 0\\
 \frac{1}{6}+\frac{f}{6}-\frac{c}{3}-\frac{d\sqrt{3}}{3}\geq 0\\
\end{array}
\right .
$$
then $(1,f,a+Ib, a-Ib, c+Id, c-Id)$ is the spectrum of a $6\times
6$ doubly stochastic matrix
\end{theorem}

 Finally, it is worth mentioning here that even applying only the
first two steps of our algorithm has an interest of its own since
it has the advantage of yielding new conditions for the (DIEP) for
the chosen dimension $k$ in Step1. To illustrate this, we include
the following example.
 \paragraph{\bf{Example 6}}
 Consider the following $4\times 4$  nonsymmetric
permutation matrix $p= \left( \begin{array}{cccc}
0& 1 & 0 & 0 \\
0 & 0 & 0 & 1 \\
1 & 0 & 0 & 0 \\
0 & 0 & 1 & 0 \\
\end{array}
\right) .$ Using Maple for example, the complex eigenvectors of
$p$ are given by: $x_1=(1,1,1,1),$ $x_2=(1,-1,-1,1),$
$x_3=(-I,1,-1,I),$ and $x_4=(I,1,-1,-I).$ Then, from these
eigenvectors, we form the complex pattern $\mathcal{S}$ matrix
$X=\left(
\begin{array}{cccc}
\frac{1}{2}& \frac{1}{2} & -\frac{I}{2} & \frac{I}{2} \\
\frac{1}{2} & -\frac{1}{2} & \frac{1}{2} & \frac{1}{2} \\
\frac{1}{2} & -\frac{1}{2} & -\frac{1}{2} & -\frac{1}{2} \\
\frac{1}{2} & \frac{1}{2} & \frac{I}{2} & -\frac{I}{2} \\
&  &  &
\end{array}
\right) .$ In addition, we let $\Pi_4$ be the diagonal matrix with
diagonal entries as $1$, $c$, $a+Ib$, $a-Ib,$ with
$trace(\Pi_4)=1+c+2a\geq 0,$   $-1\leq c \leq 1$ and $|a+Ib|\leq
1.$ Next, we check the conditions for which $X\Pi_4 X^{-1}$ is
nonnegative as follows. A simple matrix multiplication shows that
$$X\Pi_4 X^{-1}=\left(
\begin{array}{cccc}
\frac{1}{4} + \frac{c}{4}+ \frac{a}{2}&
\frac{1}{4} - \frac{c}{4}- \frac{b}{2} &
\frac{1}{4} - \frac{c}{4}+ \frac{b}{2}&
\frac{1}{4} + \frac{c}{4}- \frac{a}{2} \\
\frac{1}{4} - \frac{c}{4}+ \frac{b}{2}&
\frac{1}{4} + \frac{c}{4}+ \frac{a}{2} &
\frac{1}{4} + \frac{c}{4}- \frac{a}{2}&
 \frac{1}{4} - \frac{c}{4}- \frac{b}{2} \\
\frac{1}{4} - \frac{c}{4}- \frac{b}{2}&
\frac{1}{4} +\frac{c}{4}- \frac{a}{2} &
\frac{1}{4} + \frac{c}{4}+ \frac{a}{2}&
 \frac{1}{4} - \frac{c}{4}+\frac{b}{2} \\
\frac{1}{4} + \frac{c}{4}- \frac{a}{2}&
\frac{1}{4} - \frac{c}{4}+ \frac{b}{2} &
\frac{1}{4} - \frac{c}{4}- \frac{b}{2}&
\frac{1}{4} + \frac{c}{4}+\frac{a}{2} \\
\end{array}
\right) .$$ Note that each diagonal entry of $X\Pi_4 X^{-1}$ is
equal to $\frac{1}{4}trace(\Pi_4)$ and therefore is nonnegative.
Hence the conditions for which $X\Pi_4 X^{-1}$ is nonnegative, are
given by:
\begin{equation}
\left\{
\begin{array}
[c]{l}
\frac{1}{4} - \frac{c}{4}- \frac{b}{2}\geq 0\\
\frac{1}{4} - \frac{c}{4}+ \frac{b}{2}\geq 0 \\
\frac{1}{4} + \frac{c}{4}- \frac{a}{2}\geq 0
\end{array}
\right.
\end{equation}

Thus we have the following theorem.
\begin{theorem}
Let $1$, $c$, $a+Ib$, $a-Ib,$ be  complex numbers with $1+c+2a\geq
0,$   $-1\leq c \leq 1$ and $|a+Ib|\leq 1.$ If
$$\left\{
\begin{array}
[c]{l}
1 - c- 2b\geq 0\\
1 - c+ 2b\geq 0 \\
1 + c- 2a\geq 0
\end{array}
\right .   , $$  then $(1,c,a+Ib, a-Ib)$ is the spectrum of a
$4\times 4$ doubly stochastic matrix $D$.
\end{theorem}
Note that if we wish to have the realizing matrix $D$ has zero
trace, then it suffices to add the extra constraint
$trace(\Theta)=1+c+2a= 0$ and then we obtain the following.
\begin{theorem}
Let $1$, $c$, $a+Ib$, $a-Ib,$ be  complex numbers with $1+c+2a=
0,$ $-1\leq c \leq 1$ and $|a+Ib|\leq 1.$ If
$$\left\{
\begin{array}
[c]{l}
1 - c- 2b\geq 0\\
1 - c+ 2b\geq 0 \\
-1\leq a\leq 0
\end{array}
\right .   , $$  then $(1,c,a+Ib, a-Ib)$ is the spectrum of a
$4\times 4$ doubly stochastic matrix $D$ with zero trace.
\end{theorem}
\section*{Conclusion} We described an algorithm that yields many
sufficient conditions for three inverse eigenvalue problems
concerning doubly stochastic matrices. Although at this stage we
did not offer complete solutions, we leave it for future work to
check if it is possible to characterize all the necessary points
in Step1 needed for this algorithm to offer complete solutions at
least for low dimensions. However, besides that it offers many new
partial results, the main importance of this algorithm lies in the
fact that it can be used as a checking point in case of a
conjecture concerning complete solutions is given for any of these
three interesting problems.


\begin{thebibliography}{99}

\bibitem{b:bap}  R. B. Bapat and T. E. S. Raghavan,   Non-negative
 matrices and applications,  Cambridge University,  Cambridge, 1997.

\bibitem{i:in}  I. Bengtsson,   The importance of being
unistochastic, (2003)  arXiv:quant-ph/0403088 v1.

 \bibitem{i:ing} I. Bengtsson, \AA. Ericsson, M. Ku\'{s}, W. Tadej and K.\.{Z}yczkowski,
   Birkoff's polytope and unistochastic matrices: N=3 and N=4, (2004)  arXiv:math.CO/0402325 v2.

\bibitem{b:ber}  A. Berman and R.~J. Plemmons,  Nonnegative matrices in
the mathematical sciences,  SIAM Publications, Philadelphia, 1994.

\bibitem{b:bha}  R. Bhatia,  Matrix analysis,  Springer-Verlag,
 New York, 1997.

\bibitem{b:bru}  R. Brualdi, Some applications of
 doubly-stochastic matrices, Lin. Alg. Appl. 107, (1988) pp.~77-89.

 \bibitem{c:chu}  M. T. Chu and G. H. Golub,
  Inverse eigenvalue problems: Theory, algorithms and applications, Oxford University Press, 2005.

\bibitem{c:cru}  A. Cruse, A note on the symmetric doubly-stochastic
matrices, Discrete Mathematics, 13, (1975) pp.~109-119.

\bibitem{e:eg}  P. D. Egleston, T. D. Lenker and S. K.Narayan,
 The nonnegative inverse eigenvalue problem,  Lin. Alg. Appl.  379, (2004) pp.~475-490.

\bibitem{h:hog}  L. Hogben, Handbook of linear algebra,
 Chapman and Hall/CRC,  New York, 2007.

\bibitem{h:hor} A. Horn and C.~R. Johnson, Matrix analysis,
 Cambridge university,  Cambridge, 1981.

\bibitem{h:hw} S. G. Hwang and S. S. Pyo, The inverse eigenvalue problem for
symmetric doubly stochastic matrices, Lin. Alg. Appl., 379, (2004)
pp.~77-83.


\bibitem{j:jo} C.~R. Johnson, Row stochastic matrices similar to
doubly-stochastic matrices,  Lin. Multilin. Alg., 10, (1981) pp.~
113-130.

\bibitem{l:lo}  C.~R. Johnson, T. Laffey and R. Loewy, The real and the
symmetric nonnegative inverse eigenvalue problems are different,
Proc. Amer. Math. soc.  Vol. 124 (Number 12), (1996)
pp.~3647-3651.

\bibitem{k:ka} I. Kaddoura and B. Mourad, On a conjecture concerning the inverse
 eigenvalue problem for $4\times 4$ symmetric
doubly stochastic matrices,  Int. Math. Forum 3, 31,(2008)
pp.~1513-1519.

\bibitem{k:kat} M. Katz, On the extreme points of a certain convex
polytope, J. Combin. Theo., 8,(1970) pp.417-423.

\bibitem{l:loe} R. Loewy and D. London, A note on an inverse problem for
non-negative matrices,  Lin. Alg. Appl.  6, (1978) pp.~83-90.

\bibitem{m:min} H. Minc,  Non-negative matrices,
 Berlin Press,  New York, 1988.

\bibitem{m:mo}  B. Mourad, An inverse problem for symmetric
doubly stochastic matrices,  Inverse Problems, 19, (2003)
pp.~821-831.

\bibitem{m:mou}  B. Mourad, On a Lie-theoretic approach to
generalized doubly stochastic matrices and applications, Linear
and Multilin. Alg., 52, (2004) pp.~99-113.

\bibitem{m:mour} B. Mourad, A note on the boundary of the set
where the decreasingly ordered spectra of symmetric doubly
stochastic matrices lie, Lin. Alg. Appl., 416, (2006) pp.~546-558.

\bibitem{p:per} H. Perfect and L. Mirsky,  Spectral properties of
doubly-stochastic matrices,  Monatsh. Math.,  69, (1965)
pp.~35-57.

\bibitem{r:re}  R. Reams, An inequality for nonnegative matrices and
the eigenvalue problem, Linear and Multilin. Alg., 41, (1996)
pp.~367-375.

\bibitem{r:rea}  R. Reams,  Construction of trace zero symmetric
 stochastic matrices for the inverse eigenvalue problem, ELA 9,
 (2002) pp.~270-275.


\bibitem{s:sen}  E. Seneta,  Non-negative
 matrices and Markov chains,  2nd ed.  Springer, New York, 1981.

 \bibitem{s:sot}  R. Soto,   The inverse spectrum problem for
 positive generalized stochastic matrices,
 Computers Math. Applic., 43, (2002) pp.641-656.

\bibitem{s:sou}  G. Soules,   Constructing symmetric
 nonnegative matrices, Lin. Multilin. Alg., 13, (1983) pp.241-251.

 \bibitem{z:zy} K. \.{Z}yczkowski, M. Ku\'{s}, W. S{\l}mczyo\'{n}ski and H.-J.Sommers,
 Random unistochastic matrices,  J. Phys. A: Math.Gen. 36, (2003) pp.~3425-3450.

\end{thebibliography}
\end{document}